\documentclass[11pt]{amsart}

\usepackage{amscd}
\usepackage{amsthm}
\usepackage{amsmath}
\usepackage{amsfonts}
\usepackage{amssymb}
\usepackage[all]{xy}
\usepackage{graphicx}
\usepackage{verbatim}
\usepackage{color}
\usepackage{enumerate}
\usepackage{MnSymbol}
\usepackage{calrsfs}

\DeclareMathAlphabet{\pazocal}{OMS}{zplm}{m}{n}

\newtheorem{theorem}{Theorem}[section]
\newtheorem{lemma}[theorem]{Lemma}
\newtheorem{proposition}[theorem]{Proposition}
\newtheorem{corollary}[theorem]{Corollary}

\theoremstyle{definition}

\newtheorem{example}[theorem]{Example}

\newtheorem{remark}[theorem]{Remark}
\newtheorem{notation}[theorem]{Notation}

\usepackage{amsmath}
\usepackage{url}

\def\Jac{\mbox{\rm Jac$\:$}}

\def\O{\mathcal{O}}
\def\ff{\frak}
\def\Spec{\mbox{\rm Spec}}

\def\Max{\mbox{\rm Max}}
\def\End{\mbox{\rm End}}

\def\Zar{\mbox{\rm Zar}}

\def\patch{{\rm patch}}
\def\gen{\mbox{\rm gen}}

\def\inv{{\rm inv}}

\def\cl{\mbox{\rm cl}}
\def\cal{\mathcal}

\def\X{{\rm Zar}}

\begin{document}



\title[Compact sets and holomorphy rings]{A principal ideal theorem for compact sets  of rank one valuation rings}
 
 \date{}


\author{Bruce Olberding}

\address{Department of Mathematical Sciences, New Mexico State University,
Las Cruces, NM 88003-8001}

\maketitle

\begin{abstract} Let $F$ be a field, and let $\X(F)$ be the space of  valuation rings of $F$ with respect to the Zariski topology.    We prove that if $X$ is  a quasicompact set of 
 rank one valuation rings in $\X(F)$ whose maximal ideals do not intersect to $0$, then 
the intersection  of the rings in $X$ is an integral domain with quotient field $F$ such that every finitely generated ideal is a principal ideal. 
  To prove this result, we develop a duality between (a) quasicompact sets of rank one valuation rings whose maximal ideals do not intersect to $0$, and (b) one-dimensional Pr\"ufer domains with nonzero Jacobson radical and quotient field  $F$. 
    The necessary restriction in all these cases to collections of valuation rings whose maximal ideals do not intersect to $0$ is motivated by settings in which the valuation rings considered all dominate a given local ring. 

 \end{abstract}

\section{Introduction}

Throughout this article, $F$ denotes a field and $\X(F)$ denotes the set of valuation rings of $F$; i.e., the subrings $V$ of $F$ such that for all $0 \ne t \in F$, $t \in V$ or $t^{-1} \in V$.  
 In this article we are interested in  subrings  $A$ of $F$ which are an intersection of rank one valuation rings in a quasicompact subset of $\Zar(F)$. 
  The {\it rank} of a valuation ring, which coincides with its Krull dimension, is the real rank of its value group. Thus the rank one valuation rings have valuations  that take values in ${\mathbb{R}} \cup \{\infty\}$.
The {\it Zariski topology} on $\X(F)$ is the topology having as a basis of open sets the subsets of $\X(F)$ of the form $\{V \in \X(F):t_1,\ldots,t_n \in F\}$ for $t_1,\ldots,t_n \in F$. With this topology, $\X(F)$ is the {\it Zariski-Riemann space of $F$}.  The main purpose of this article is to prove the following 
instance of what Roquette \cite{Roq} calls a 
Principal Ideal Theorem, that is,  a theorem which guarantees a given class of integral domains has the property that every finitely generated ideal  is principal. 

\medskip

{\noindent}{\bf Main Theorem.} {\it  If $X$ is a  quasicompact set of rank one  valuation rings in $\X(F)$ whose maximal ideals do not intersect to $0$, then the intersection of the valuation rings in $X$ is an integral domain with quotient field $F$ and Krull dimension one such that  every finitely generated ideal  is a principal ideal.  }

\medskip


The theorem, which most of the paper is devoted to proving, asserts that the intersection of such valuation rings is  a {\it B\'ezout domain}, a domain for which every finitely generated ideal is principal. Such rings belong to  the extensively studied  class of {\it Pr\"ufer domains}, those domains $A$ for which $A_M$ is a valuation domain for each maximal ideal $M$ of $A$.  
  The problem of when an intersection of valuation rings is a Pr\"ufer domain is a difficult but well-studied problem that has applications to real algebraic geometry (e.g., \cite{Bec82, Ber95, BK89, Sch82b}), non-Noetherian commutative ring theory (e.g., \cite{Gil, LT, Rus}), formally $p$-adic holomorphy rings \cite{Roq} and the study of integer-valued polynomials \cite{CC, Lop}. 
      Using the equivalence of this problem to that of 
     determining when a subspace of $\X(F)$ yields an affine scheme, a geometric criterion involving morphisms of the Zariski-Riemann space into the projective line was given in \cite{OGeo}.  
     In that approach, treating the Zariski-Riemann space as a locally ringed space, not simply a topological space, is crucial.  
 By contrast, the main theorem  shows that unlike in the general case, the geometry of the  locally ringed space structure is not needed to distinguish 
a B\'ezout intersection   when the valuation rings satisfy the hypotheses of the theorem. Instead, the question of whether the intersection  of such a collection of rank one valuation rings  is a Pr\"ufer domain is purely topological. 

 Moving beyond rank one valuation rings,   quasicompactness of a subset $X$ of $\X(F)$ is far from sufficient to guarantee that the intersection of valuation rings in $X$ is a B\'ezout domain. 
 For example, if $D$ is a Noetherian local domain with quotient field $F$,  
  then for any $t_1,\ldots,t_n \in F$,  $$ {\cal U}= \{V \in \X(F/D):x_1,\ldots,x_n \in V\}$$ is quasicompact, but  the intersection of the valuation rings in ${\cal U}$ is a B\'ezout domain if and only if the integral closure of $D[x_1,\ldots,x_n]$ is a principal ideal domain, something that occurs only for very special choices of $D$ and $x_1,\ldots,x_n$. In fact, 
the main  theorem is optimal in the sense that if any one of the hypotheses ``quasicompact'', ``rank one'', or ``the maximal ideals do not intersect to $0$'' is omitted, the conclusion is false; see Example~\ref{main example}. Note that the condition that the maximal ideals of the rank one valuation rings in $X$ do not intersect to $0$ occurs naturally in settings where the valuation rings in $X$ are assumed to dominate a local ring of Krull dimension $>0$. Examples of such settings of recent interest include Berkovich spaces and tropical geometry (see for example \cite{GG}) and valuative trees of regular local rings  \cite{FJ, Granja}.


 Interest in compactness in the Zariski-Riemann space  dates back to Zariski's introduction of the topology on $\Zar(F)$ in \cite{ZCom}. 
 %
If $D$ is a subring of $F$, then $\X(F/D)$, the subspace of $\X(F)$ consisting of the valuation rings in $\X(F)$ that contain $D$ as a subring, is the {\it Zariski-Riemann space of $F/D$}. 
That  $\X(F/D)$ is quasicompact was proved by Zariski \cite{ZCom} in 1944 as a step in  his  program for resolution of singularities of surfaces and three-folds. In more recent treatments of the topology of $\X(F/D)$ such as \cite{FFL, OZR},
 quasicompactness is viewed as part of a more refined topological picture that treats  $\X(F/D)$ as a spectral space 
  or as a locally ringed space that is a projective limit of projective schemes. The latter  point of view also has its origins in Zariski's work  \cite{ZCom}.  

 However, in restricting to the subspace of rank one valuation rings of $F$, key topological  features of $\X(F)$ are lost. For example, this subspace need not be  spectral, nor even quasicompact.  
Yet, in passing to the space of rank one valuation rings, the main theorem shows  that the topology becomes much more consequential for the ring-theoretic structure of an intersection of valuation rings.  One of the key steps in proving this is first showing that in the setting of the main theorem, the intersection of valuation rings in $X$  is a Pr\"ufer domain. 
   We prove this in Theorem~\ref{correspond 1} by establishing the following lemma. 
 For a subset $X $ of $\X(F)$, we let $A(X) = \bigcap_{V \in X}V$ be the {\it holomorphy ring}\footnote{See Roquette \cite{Roq} for an explanation of this terminology.} of $X$ and $J(X) = \bigcap_{V \in X}{\ff M}_V$ be the ideal of $A(X)$ determined by the intersection of the  maximal ideals ${\ff M}_V$ of the valuation rings $V$.  
 For a ring $A$, $\Max(A)$ denotes its set of maximal ideals. 

\medskip

{\noindent}{\bf Main Lemma.} {\it The mappings  \begin{center} $X \mapsto A(X)$ \: and \:  $A \mapsto \{A_M:M \in \Max(A)\}$\end{center}
 define a bijection between
  the quasicompact sets  $X$ of rank one valuation rings in  $\X(F)$ with $J(X) \ne 0$ and 
the one-dimensional Pr\"ufer  domains $A$ with nonzero Jacobson radical and  quotient field $F$. }

\medskip

Corollary~\ref{correspond 2} gives  another version of this result in which the spaces $X$ need not be assumed {\it a priori} to satisfy $J(X) \ne 0$. In this case, ``quasicompact'' is replaced with ``compact'' ($=$ quasicompact and Hausdorff), and the holomorphy rings are all assumed to have quotient field $F$. 
    A consequence of the main lemma  is that if $X$ is quasicompact, $J(X) \ne 0$ and $X$ consists of rank one valuation rings, then  
$X = \{A(X)_M:M \in \Max(A(X))\}$.   
  As this suggests, the main lemma can be recast in the language of schemes, and we do this in Corollary~\ref{scheme}.

  The applicability of the main theorem  depends on whether a space of rank one valuation rings can be determined to be quasicompact. The key technical observation behind our approach  is that if a subset  $X$ of $\X(F)$ consists of rank one valuation rings and $J(X) \ne 0$, then   $X$ is quasicompact if and only if $X$ is  closed in the patch topology of $\X(F)$. (See Section 2.)  On the level of proofs and examples, this reduces the issue of quasicompactness  to calculation of patch limits point of $X$ in $\X(F)$, and specifically whether such  limit points have rank one. 
  In a future paper \cite{ODiv}  we use the results of the present article along with additional methods to show how to apply these ideas to divisorial valuation overrings of a Noetherian local domain of Krull dimension two. This is discussed in Remark~\ref{last remark}.
  
  The present paper is motivated by recent work such as in  \cite{FFL, FS, OIrr, ONoeth, OGeo, OSpringer, OGraz, OZR}  on understanding 
 how topological or geometric properties of a space of valuation rings are reflected in the algebraic structure of the intersection of these valuation rings.

  


\section{Topological preliminaries}



In this section we outline the topological point of view needed for the later sections. Recall that throughout the paper, $F$ denotes an arbitrary  field. 

\begin{notation} {For each subset $S$ of the field $F$, let \begin{center} ${\cal U}(S) = \{V \in \X(F):S \subseteq V\}$ \: and \: ${\cal V}(S) = \{V \in \X(F):S \not \subseteq V\}$. \end{center} }
\end{notation}

The {\it Zariski topology} on $\X(F)$ has as a basis of nonempty open sets the sets of the form ${\cal U}(x_1,\ldots,x_n)$, where  $x_1,\ldots,x_n \in F$. The set $\X(F)$ with the Zariski topology is the {\it Zariski-Riemann space} of $F$.  
Several authors have established independently that the Zariski-Riemann space $\X(F)$ is a {\it spectral space}\footnote{The terminology is motivated by a theorem of Hochster \cite{Hoc} that shows a space is spectral if and only if it is homeomorphic to the prime spectrum of a ring}, meaning that (a) $\X(F)$ is quasicompact and $T_0$, (b) $\X(F)$ has a basis of quasicompact open sets, (c) the intersection of finitely many quasicompact open sets in $\X(F)$ is quasicompact, and (d) every nonempty irreducible closed set in $\X(F)$ has a unique generic point. 
 See \cite{FFL} and \cite{OZR} for more on the history of this central result.  A spectral space admits the {\it specialization order} $\leq $ given by $x \leq y$ if and only if $y$ is in the closure of $\{x\}$. In the Zariski  topology on $\X(F)$, we have for $V , W \in \X(F)$ that $V \leq W$ if and only if $W \subseteq V$. We use the specialization order in the results in this section but not elsewhere in the paper.

  As a spectral space, $\X(F)$ admits two other useful topologies, the inverse and patch topologies.  
 The {\it inverse topology} on $\X(F)$ is the topology  that has, as a  basis of closed sets, the subsets of $\X(F)$ that are quasicompact and open in the Zariski topology; i.e., the nonempty closed sets are intersections of finite unions of sets of the form $\cal{U}(x_1,\ldots,x_n)$, $x_1,\ldots,x_n \in F$.   The inverse topology is useful for dealing with issues of irredundance and uniqueness of representations of integrally closed subrings of $F$; for example, see \cite{OGraz}.  In the present article, we use the inverse topology in a limited way.

The most important topology on $\X(F)$ for the purposes of this article is the  {\it patch topology}
on $\X(F)$, which is given by the topology that  has as a basis of open sets the subsets of $\X(F)$  of the form  \begin{center}
${\cal U}(x_1,\ldots,x_n), $ \:  ${\cal V}(y_1)$ \:  or \:  ${\cal U}(x_1,\ldots,x_n)  \cap {\cal V}(y_1) \cap \cdots \cap {\cal V}(y_m)$,
\end{center}
where $ x_1,\ldots,x_n,y_1,\ldots,y_m \in F$. The complement in $\X(F)$ of any set in this basis is again open in the patch topology. Thus  the patch topology has as a basis sets that are both closed and open  (i.e., the patch topology is zero-dimensional). The  patch topology is also spectral and hence  quasicompact
 \cite[p.~45]{Hoc}.    Unlike the Zariski and inverse topologies on $\X(F)$, the patch topology  is always  Hausdorff \cite[Theorem~1]{Hoc}.  
 
 \medskip
 
{\bf Convention.} In the article we work with all three topologies, inverse, patch and Zariski, sometimes even in the same proof. To avoid confusion, we insert the adjective ``patch'' before a topological property when working with it in the patch topology. For example, a ``patch open set'' is a set that is open in the patch topology. Similarly, we insert ``inverse'' as an adjective when working with the inverse topology. If no adjective is present (e.g., ``the set $Z$ is quasicompact''), this is always to be understood as indicating we are working in the Zariski topology. Thus the Zariski topology is the default topology if no other topology is specified. Recall also from the Introduction that by {\it compact} we mean both quasicompact and Hausdorff.

\medskip

One of our main technical devices in the paper is that of a patch  limit point. 
  Let $X \subseteq \X(F)$.  Then $V \in \X(F)$ is a {\it patch limit point of $X$} if each patch open neighborhood ${\cal U}$ of $V$ in $X$ contains a point in $X$ distinct from $V$; equivalently (since the patch topology is Hausdorff), every patch open neighborhood ${\cal U}$  
  of $V$ contains infinitely many valuation rings in $X$.   Applying the relevant definitions, it follows that $V$ is a patch limit point of $X$ if and only if for all finite (possibly empty) subsets $S$ of $V$ and $T$ of ${\ff M}_V$ there is a valuation ring $U$ in $X \setminus \{V\}$ such that $S \subseteq U$ and $T \subseteq {\ff M}_U$.\footnote{Throughout the paper, we denote the maximal ideal of a valuation ring $V$ by ${\ff M}_V$.}

\begin{notation} \label{nota} Let $\emptyset \ne X \subseteq \X(F)$. We use the following notation.

\begin{enumerate}[$(1)$] 
\item[$\bullet$] $\lim(X) = $ the set of patch limit points of $X$ in $\X(F)$. 
\item[$\bullet$] $\patch(X) = X \cup \lim(X) = $ closure of $X$ in the patch topology of $\X(F)$.
\item[$\bullet$] $A(X) = \bigcap_{V \in X}V = $ holomorphy ring of $X$. 
\item[$\bullet$] $J(X) = \bigcap_{V \in X}{\ff M}_V$.  


\end{enumerate}
\end{notation}



In the next lemma we collect some properties of patch closure that are needed in later sections.  
 More systematic treatments of Zariski, patch and inverse closure in $\X(F)$ can be found  in \cite{FFL} and \cite{OZR} and their references.    
 
\begin{lemma} \label{list} Let $X$ be a nonempty subset of $\X(F)$.  Then
\begin{enumerate}[$(12)$]
\item[{\em (1)}] $A(X)  = \bigcap_{V \in \patch(X)}V$ and $J(X) = \bigcap_{V \in \patch(X)}{\ff M}_V.$

\item[{\em (2)}] The set $\patch(X)$ is spectral in the subspace Zariski topology.

\end{enumerate}
Suppose in addition that $X$ is quasicompact and consists of rank one valuation rings.  
\begin{enumerate}[$(12)$] 
\item[{\em (3)}] 
The set $\patch(X)$ is contained in  $ X \cup \{F\}$. 



\item[{\em (4)}]  If $S$ is a multiplicatively closed subset  of $A  = A(X)$, then $A_S = \bigcap_{A_S \subseteq V \in X \cup\{F\}}V$.

\end{enumerate}
\end{lemma} 

\begin{proof} 
(1) The first assertion   can be found in \cite[Proposition 4.1]{FFL} or \cite[Proposition 5.6]{OZR}. Since $\patch(X) = X \cup \lim(X)$, to
 see that  the second assertion holds, it suffices to show that $J(X) \subseteq {\ff M}_U$ for each $U \in \lim(X)$.  Let $0 \ne a \in J(X)$, and let $U \in \lim(X)$. If $a \not \in {\ff M}_U$, then $U \in {\cal U}(a^{-1})$. Since $U \in \lim(X)$, there exists $V \in X \cap {\cal U}(a^{-1})$, so that $a^{-1} \in V$. However, $a \in {\ff M}_V$ by the choice of $a$, a contradiction. Thus $J(X) \subseteq {\ff M}_U$, which verifies (1).  

(2) A patch closed subspace of a spectral space is spectral in the subspace topology \cite[Proposition 9.5.29, p.~433]{Gou}. 
%
  
  (3) Suppose $X$ is quasicompact. This implies that if $U \in \patch(X)$, then there exists $V \in X$ such that $V \subseteq U$ \cite[Proposition 2.2]{OZR}. Since $X$  consists of rank one valuation rings, we have $V = U$ or $U = F$. Thus $\patch(X) \subseteq  X \cup \{F\}$. 

(4) By \cite[Corollary 5.7]{OZR}, $$A_S \: \: = \: \: \bigcap_{A_S \subseteq V \in Y}V,$$ where $Y$ is the set of all $V \in \X(F)$ such that $V \supseteq U$ for some $U \in \patch(X)$. Thus (4) follows from (3) and the fact that every valuation ring in $X $ has rank one. 
\end{proof}






The following proposition  reinterprets for rank one valuation rings  the property of compactness in the Zariski topology of $\X(F)$ in terms of the patch topology. This enables us to work with the patch topology-- and specifically, patch limit points-- in the algebraic arguments of the next sections. The proposition also shows that 
the Hausdorff condition for a set $X$ of rank one valuation rings is closely connected with the algebraic property that $J(X) \ne 0$.  

\begin{proposition} \label{Haus lemma}  \label{T2} The following statements hold for every  nonempty set $X$ of rank one valuation rings  in $\Zar(F)$.   
 
 \begin{enumerate}[$(12)$] 
 
 \item[$(1)$] $X$ is  compact  if and only if $X$ is patch closed in $\Zar(F)$. 
 
 \item[$(2)$]    If $J(X) \ne 0$, then 
  the Zariski and patch topologies agree on $X$, and hence $X$ is Hausdorff and zero-dimensional.
  
  \item[$(3)$] Suppose  $A(X)$ has quotient field $F$. Then $J(X) \ne 0$ if and only if   $X$ is Hausdorff.  

\end{enumerate}
\end{proposition}

\begin{proof} 
(1)  Suppose that $X$ is compact. 
 By Lemma~\ref{list}(3),    $\patch(X) \subseteq X \cup \{F\}$, so to show that $X$ is patch closed it suffices to show that $F \not \in \patch(X)$.  
 If $X$ consists of a single valuation ring, then $\patch(X) = X$ and  
the claim is clear since this valuation ring must have rank one. Suppose $X$    
contains at least two distinct valuation rings $V$ and $W$.  Since $X$ is Hausdorff, there exist $x_1,\ldots,x_n,y_1,\ldots,y_m \in F$ such that $V \in {\cal U}(x_1,\ldots,x_n)$, $W \in {\cal U}(y_1,\ldots,y_m)$ and $${\cal U}(x_1,\ldots,x_n) \cap {\cal U}(y_1,\ldots,y_m) \cap X = \emptyset.$$ Let $$Y = {\cal V}(x_1,\ldots,x_n) \cup {\cal V}(y_1,\ldots,y_m).$$ Then $X \subseteq Y$ and $Y$ is patch closed, so $\patch(X)  \subseteq Y$. Since $F \not \in Y$, we have $F \not \in \patch(X)$. Therefore, $X$ is patch closed.  

Conversely, suppose that $X$ is patch closed in $\X(F)$. By Lemma~\ref{list}(2), $X$ is a spectral space with respect to the Zariski topology and hence is quasicompact. Since the valuation rings in $X$ have rank one and $F \not \in X$, the elements of $X$ are minimal in $X$  with respect to the specialization order. This implies  the patch and Zariski topologies agree on $X$ \cite[Corollary 2.6]{ST}.  Thus $X$ is Hausdorff since the patch topology is Hausdorff. 

(2)  Let $ 0 \ne x \in J(X)$. Then $X \subseteq {\cal V}(x^{-1})$.   Since the valuation rings in $X$ have rank one and $F \not \in {\cal V}(x^{-1})$, the elements of $X$  are minimal in ${\cal V}(x^{-1})$ with respect to the specialization order induced by the Zariski topology. As a patch closed subset of $\X(F)$,  the set ${\cal V}(x^{-1})$ is, with respect to the Zariski topology, a spectral space (Lemma~\ref{list}(2)). Thus the Zariski and patch topologies agree on the set of elements of ${\cal V}(x^{-1})$ that are minimal with respect to the specialization order \cite[Corollary~2.6]{ST}. Since the patch topology is Hausdorff and zero-dimensional, the last statement of (2) now follows.


(3) Suppose $A = A(X)$ has quotient field $F$ and  $X$ is Hausdorff. If $X$ consists of a single valuation ring, then it is clear that $J(X) \ne 0$. Assume that $X$ has more than one valuation ring.  Since $X$ is Hausdorff and $A$ has quotient field $F$,   there exist nonzero $a_1,\ldots,a_n, b_1,\ldots,b_m,c \in A$ such that 
 $${\cal U}(a_1/c,\ldots,a_n/c) \cap {\cal U}(b_1/c,\ldots,b_m/c) \cap X = \emptyset.$$  
Therefore,  ${\cal U}(1/c) \cap X = \emptyset$. Since each $V \in X$ is a valuation ring, this implies $c \in {\ff M}_V$, so that $0 \ne c \in J(X) $. The converse follows from (2).
\end{proof}

\section{Residually transcendental limit points}

The main result of this section, Theorem~\ref{trans}, is of a technical nature and  involves the existence of patch limit points that are residually transcendental over a local subring of $F$.  
   For the purpose of proving the results in this section,
we recall the notion of a projective model of a field; see \cite[Chapter 6, \S 17]{ZS} for more background on this topic. Let $D$ be a subring of the field $F$, and let $t_0,t_1,\ldots,t_n$ be nonzero elements of $F$. For each $i$, let $D_i =D[t_0/t_i,\ldots,t_n/t_i]$, and let $${\pazocal{M}} \: = \: \bigcup_{i=0}^n \: \{(D_i)_P: P \in \Spec(D_i)\}.$$ Then ${\pazocal{M}}$ is {\it the projective model of $F/D$ defined by $t_0,t_1,\ldots,t_n$.}    Alternatively, a projective model of $F/D$ can be viewed as a projective integral scheme over $\Spec(D)$ whose function field is a subfield of $F$. Motivated by this interpretation, it is often convenient to view ${\pazocal{M}}$ as a locally ringed space; see for example \cite[Section 3]{OZR}.  We do not need to do this explicitly in the present paper, but the notation remains helpful here when we wish to view the local rings in ${\pazocal{M}}$ as points. In particular, viewing the set ${\pazocal{M}}$ as a topological space with respect to the Zariski topology\footnote{The basis for this topology is given by sets of the form $\{R \in {\pazocal{M}}:x_1,\ldots,x_n \in R\}$, where $x_1,\ldots,x_n \in F$.}, the local rings in ${\pazocal{M}}$ are points in ${\pazocal{M}}$.  For this reason, and in keeping with the locally ringed space point of view, we denote a local ring $x \in {\pazocal{M}}$   by $\O_{{\pazocal{M}},x}$, despite the redundancy in doing so.    
A subset $Y$ of ${\pazocal{M}}$ is an {\it affine submodel} of ${\pazocal{M}}$ if there exists a  $D$-subalgebra $R$ of $F$ such that $Y = \{R_P:P \in \Spec(R)\}$. (We differ here from \cite{ZS} in that we do not require $R$ to be a finitely generated $D$-algebra.)   

For each $x \in {\pazocal{M}}$, there exists a valuation ring in $\Zar(F)$ that dominates $\O_{{\pazocal{M}},x}$, and for each $V \in \Zar(F/D)$, there exists a unique $x \in {\pazocal{M}}$ such that $V$ dominates $\O_{{\pazocal{M}},x}$; see \cite[pp.~119--120]{ZS} or apply the valuative criterion for properness \cite[Theorem~4.7, p.~101]{Hart}. 
 For a nonempty subset $X$ of ${\pazocal{M}}$, we denote by ${\pazocal{M}}(X)$ the set of all $x \in {\pazocal{M}}$ such that $\O_{{\pazocal{M}},x}$ is dominated by some $V \in X$.  More formally, ${\pazocal{M}}(X)$ is the image of $X$ under the continuous closed surjection $\X(F) \rightarrow {\pazocal{M}}$ given by the domination mapping \cite[Lemma 5, p.~119]{ZS}.

\begin{lemma} \label{principalization}  Let $X$ be a nonempty subset of $\X(F)$ such that $A=A(X)$ is a local ring. 
Let $D$ be a subring of $A$, let $t_0,\ldots,t_n$ be nonzero elements of $F$, and let ${\pazocal{M}}$ be 
  the projective model of $F/D$ defined by $t_0,\ldots,t_n.$ If ${\pazocal{M}}(X)$ is finite, then $t_0/t_i,\ldots,t_n/t_i \in A$ for some $i = 0,1,\ldots,n$.  \end{lemma} 

\begin{proof}    Suppose ${\pazocal{M}}(X)$ is finite.  
Each finite subset of the projective model ${\pazocal{M}}$ is contained in an affine open submodel of ${\pazocal{M}}$. (This  is a consequence of homogeneous prime avoidance; see for example  \cite[Lemma 01ZY]{stacks} or the proof of  \cite[Corollary~3.2]{OGeo}). Thus  
${\pazocal{M}}(X)$ is contained in an affine submodel of ${\pazocal{M}}$, and hence there is a $D$-subalgebra $R$ of $F$ such that for each $x \in {\pazocal{M}}(X)$, $\O_{{\pazocal{M}},x}$ is a localization of $R$. 
For each $V \in X$, there is $x \in {\pazocal{M}}(X)$ such that $V$ dominates
 $\O_{{\pazocal{M}},x}$. Since $A$ is a local ring, so is the subring $$B  \:\: = \:\bigcap_{x \in {\pazocal{M}}(X)}\O_{{\pazocal{M}},x}$$ of $A$. (Indeed, if $b \in B$ is a unit in $A$, then since, for each $x \in {\pazocal{M}}(X)$, $\O_{{\pazocal{M}},x}$ is dominated by a valuation ring in $X$, $b$ is a unit in $B$.  Thus if $b,c \in B$ are nonunits in $B$, then $b+c$ is a nonunit in the local ring $A$, so that $b+c$ is a nonunit in $B$.) 
  For each $x \in {\pazocal{M}}(X)$,  since  $\O_{{\pazocal{M}},x}$ is a localization of $R$ and $R \subseteq B \subseteq \O_{{\pazocal{M}},x}$ with $B$ a local ring, we have  that $\O_{{\pazocal{M}},x}$ is a localization of $B$ at a prime ideal of $B$.  
 Since ${\pazocal{M}}(X)$ is finite,  $B$ is a local ring 
that is  a finite intersection of the local rings $\O_{{\pazocal{M}},x}$, $x \in {\pazocal{M}}(X)$, each of which is a localization of $B$ at a prime ideal. It follows that $B$ is equal to one of these localizations; i.e., $B = \O_{{\pazocal{M}},x}$ for some $x \in {\pazocal{M}}(X)$.   Since $B \in {\pazocal{M}}$, there is $i$ such that $t_0/t_i,\ldots,t_n/t_i \in B \subseteq A$.  
\end{proof}


The proof of the next lemma can be streamlined by using the language of schemes and morphisms into the projective line (see \cite{OGeo} for more on this point of view in our context), but in order to make the proof self-contained, we develop the needed ideas in the course of the proof.

\begin{theorem} \label{trans} Let $X$ be a nonempty subset of $\X(F)$ such that $A=A(X)$ is a local ring,  and let $0 \ne t \in F$.  Suppose  $D$ is  a local  subring of $A$,
  $t \not \in A$, $1/t \not \in A$,   
and 
  all but at most finitely many $V \in X$ dominate $D$.  
 Then there exists $U \in \lim(X)$ such that $t,1/t \in U$, $U$ dominates $D$ and  the image of $t$ in $U/{\ff M}_U$ is transcendental over the residue field of $D$. 

\end{theorem}

\begin{proof}  We first show that we can replace $D$ with a local subring $D'$ of $A$ such that $D'$ is integrally closed in $D'[t]$ and all but finitely many valuation rings in $X$ dominate $D'$.  
 Let $\overline{D}$ denote the integral closure of $D$ in $D[t]$. Since $A$ in integrally closed in $F$, $\overline{D} \subseteq A$. Let ${\ff m} = M \cap \overline{D}$, where $M$ is the maximal ideal of $A$, and let $D' = \overline{D}_{\ff m}$.  Then $A$ dominates $D'$. 
 Since $\overline{D} \subseteq D[t]$, we have $\overline{D}[t] = D[t]$, and so $\overline{D}$ is integrally closed in $\overline{D}[t]$. Thus $D' = \overline{D}_{\ff m}$ is integrally closed in $D'[t] = \overline{D}_{\ff m}[t]$  since $D'$ and $D'[t]$ are localizations of $\overline{D}$ and $\overline{D}[t]$, respectively, at the same multiplicatively closed set. Therefore, $D'$ is a local ring that is  integrally closed  in $D'[t]$ and is dominated by $A$.      
 

To see next that all but finitely many valuation rings in $X$ dominate $D'$, suppose that 
$V \in X$ and $V$ dominates $D$. We claim that $V$ dominates $D'$. Let ${\ff n}$ denote the maximal ideal of $D$, and let ${\ff p} = {\ff M}_V \cap \overline{D}$.  Since $V$ dominates $D$,  ${\ff p}$ lies over ${\ff n}$.  Thus, since $\overline{D}$ is integral over $D$,  ${\ff p}$ is a maximal ideal of  $\overline{D}$. 
If ${\ff p} \ne {\ff m}$, then there exists $d \in {\ff p} \setminus {\ff m}$ so that $1/d \in \overline{D}_{\ff m} = D' \subseteq A \subseteq V$. Since $d \in {\ff M}_V$, this is  a contradiction that implies ${\ff p} = {\ff m}$.  
Thus $V$ dominates $D'$. This shows that every valuation ring in $X$ that dominates $D$ also dominates $D'$. Thus, if $V \in X$ does not dominate $D'$, then $V$ does not dominate $D$.  By assumption, there are at most finitely many valuation rings in $X$ that do not dominate $D$, so there are at most finitely many valuation rings in $X$ that do not dominate $D'$.   With this in mind, we work for the rest of the proof with $D'$ instead of $D$ and draw the desired conclusion for $D$ in the last step of the proof. The advantage of working with $D'$ is  that $D'$ is a local ring that is integrally closed in $D'[t]$.

 Let ${\pazocal{M}}$ be the projective model of $F/D'$ defined by $1,t$.  
Let $f:{\pazocal{M}} \rightarrow \Spec(D')$ be the canonical mapping that sends a local ring $R$ in $ {\pazocal{M}}$ to $D' \cap N$, where $N$ is the maximal ideal of $R$. Let $C = f^{-1}({\ff m}')$, where ${\ff m}'$ is the unique maximal ideal of $D'$.

 Since $f$ is continuous in the Zariski topology \cite[Lemma 5, p.~119]{ZS},  $C$ is the   closed subset of ${\pazocal{M}}$    consisting of all the local rings in ${\pazocal{M}}$ that dominate $D'$.    Since $D'$ is integrally closed in $D'[t]$ and $t,t^{-1} \not \in D'$ (indeed, by assumption, $t,t^{-1} \not \in A$),  Seidenberg's Lemma \cite[Exercise 31, pp.~43--44]{Kap}  implies that the rings  $D'[t]/{\ff m}D'[t]$ and $D'[t^{-1}]/{\ff m}'D'[t^{-1}]$ are each isomorphic to the polynomial ring $(D'/{\ff m}')[T]$, where $T$ is an indeterminate. These isomorphisms are induced by the mappings $t \mapsto T$ and $t^{-1} \mapsto T$, respectively. 
 Thus $(t,{\ff m}')D'[t]$ is a maximal ideal of $D'[t]$, while ${\ff m}'D'[t]$ is a prime ideal in $D'[t]$ and ${\ff m}'D'[t^{-1}]$ is a prime ideal in $D'[t^{-1}]$.

Now,  since  ${\ff m}'$ extends to a prime ideal in both $D'[t]$ and $D'[t^{-1}]$,  $C$ is irreducible in the Zariski topology with generic point the local ring $D'[t]_{{\ff m}'D'[t]} = D'[t^{-1}]_{{\ff m}'D'[t^{-1}]}$. 
Using the fact   that  the rings  $D'[t]/{\ff m}'D'[t]$ and $D'[t^{-1}]/{\ff m}'D'[t^{-1}]$ are each PIDs isomorphic to the polynomial ring $(D'/{\ff m}')[T]$, it follows that
the  closed points in $C$ are the local rings in the set 
$$\{D'[t]_{P}:P {\mbox{ is a maximal ideal in }} D'[t], {\ff m} \subseteq P\} 
 \cup
\{D'[t^{-1}]_{({\ff m},t^{-1})}\}.$$  
The only point in $C$ that is not  closed is $D'[t]_{{\ff m}'D'[t]}$, the unique generic point of $C$. This accounts for all the local rings in $C$.  The rest of the proof consists in showing that there is a valuation ring $U$ in $\lim(X)$ that dominates $D'[t]_{{\ff m}'D'[t]}$.

To this end, we next describe the local rings in ${\pazocal{M}}(X) \cap C$; i.e., we describe the local rings in $C$ that are dominated by valuation rings in $X$. In particular, we show that there are infinitely many such local rings in $C$ and 
that the Zariski closure of ${\pazocal{M}}(X) \cap C$ in ${\pazocal{M}}$ is  $C$.  

Let $X^*$ be the set of valuation rings in $X$ that dominate $D'$.
We have established that all but finitely many valuation rings in $X$  dominate $D'$; that is, 
$X \setminus X^*$ is finite.  The image ${\pazocal{M}}(X^*)$ of $X^*$ in ${\pazocal{M}}$ under the domination mapping is contained in $C$, so that   $C(X^*) = {\pazocal{M}}(X^*) $. By Lemma~\ref{principalization} the fact that  $A$ is a local ring and  $t$ and $1/t$ are not elements of $ D'$  
implies ${\pazocal{M}}(X)$ is infinite. Since also $X \setminus X^*$ is finite, it follows that ${\pazocal{M}}(X^*) = C(X^*)$ is infinite. Thus, since $C$ consists only of   closed points and a unique generic point for $C$, there are infinitely many   closed points of $C$ in $C(X^*)$, which means there is a subset $X'$ of $X^*$ such that the image $C(X')$ of $X'$ in $C$ is infinite and consists of rings of the form $D'[t]_P$, where $P$ is a maximal ideal of $D'[t]$ that contains ${\ff m}'$. Therefore,  the valuation rings in $X'$ contain $D'[t]$, and the image of $X'$ in $\Spec(D'[t])$ under the map that sends a valuation ring to its center in $D'[t]$ consists of infinitely many maximal ideals of $D'[t]$, all of which contain the dimension one prime ideal ${\ff m}'D'[t]$.  
 Since $D'[t]/{\ff m}'D'[t]$ is a PID, the intersection of these infinitely many maximal ideal  is ${\ff m}'D'[t]$.

Next, to see how this fact is reflected in $C$, let ${\pazocal{M}}_1$ be the affine submodel of ${\pazocal{M}}$ given by $${\pazocal{M}}_1 = \{D'[t]_P:P \in \Spec(D'[t])\}.$$  Let $g:{\pazocal{M}}_1
 \rightarrow \Spec(D'[t])$ denote the homeomorphism that sends  a local ring $D'[t]_P$ in ${\pazocal{M}}_1$ to its center in $D'[t]$. We have shown that the Zariski closure of $C(X')$ in $\Spec(D'[t])$ is the set of all prime ideals of $D'[t]$ containing the prime ideal ${\ff m}'D'[t]$.    Since $g$ is a homeomorphism, it follows  that $D'[t]_{{\ff m}'D'[t]}$ is in the Zariski closure of $C(X')$. Hence the Zariski closure of $C(X')$ in ${\pazocal{M}}$ is $C$.

 Now, since  the generic point of the   closed set $C$ is the local ring $D'[t]_{{\ff m}'D'[t]}$ and $C(X')$ is infinite and  Zariski dense in $C$, we apply \cite[Lemma 2.7(4)]{OZR} to obtain that
 there 
 is a valuation ring $U$  in   the patch closure of $X'$   
 that dominates $$D'[t]_{{\ff m}'D'[t]} = D'[1/t]_{{\ff m}'D'[1/t]}.$$ 
  Since $$D'[t]/{\ff m}'D'[t] \cong (D'/{\ff m}')[T],$$ the 
  image of $t$ in the residue field of $U$  is transcendental over the residue field of $D'$. Since all the valuation rings in $X'$ are centered  on maximal ideals of $D'[t]$, $U$ is not a member of $X'$. Therefore, since $U$ is in the patch closure of $X'$ but  not in $X'$, it must be that $U \in \lim(X') \subseteq \lim(X)$. Finally, 
 since $D'$ dominates $D$, we conclude that the image of  $t$ in $U/{\ff M}_U$ is  transcendental  over the residue field of $D$, 
  which completes the proof of the theorem.
\end{proof}

\section{Limit points of rank greater than one}

We show in Theorem~\ref{ugh} that if $X \subseteq \X(F)$ and $ A(X)$ is local but not a valuation domain, then there is a patch limit valuation ring  of $X$ of rank $>1$. This is the key result  needed in the next section to prove the main results of the paper.  
The proof of Theorem~\ref{ugh} relies on the following lemma, which gives a criterion for the existence of a  valuation ring of rank $>1$ to lie in the patch closure of $X$.   

\begin{lemma} \label{step} Let $A$ be a local integrally closed   subring of $F$ with maximal ideal $M$, and let $X$ be a patch closed subset of $\X(F)$ such that $A = A(X)$.  Suppose that there exist $0 \ne t \in F$ and $m \in M$ such that $mt \not \in A$ and $t^{-1} \not \in A$. If for each $i>0$ there exists $V_i \in X$ such that $m \in {\ff M}_{V_i}$  and $mt^i$ is a unit in $V_i$, then $X$ contains a valuation ring of rank $>1$. 
\end{lemma}

\begin{proof} 
 Let $i > 0$. If $mt^i \in A$, then $(mt)^i \in A$, so that since $A$ is integrally closed we have  $mt \in A$, contrary to assumption. Thus $mt^i \not \in A$ for all $i>0$.  Moreover,  
  $(mt^{i})^{-1} \not \in A$, since otherwise $t^{-i} = m(m^{-1}t^{-i}) \in A$, which since $A$ is integrally closed forces $t^{-1} \in A$,  contrary to the choice of $t$.   
  By assumption, there exists $V_i \in X$ such that $mt^i$ is a unit in $V_i$ and $m \in {\ff M}_{V_i}$.   
 If $t \in V_i$, then since $mt^i$ is a unit in $V_i$ it is the case that $m$ is a unit in $V_i$ and an element of  ${\ff M}_{V_i}$, a contradiction. Thus  $t \not \in V_i$.  Using Notation~\ref{nota}, let
 $${\cal C}_i= {\cal U}(mt^i) \cap {\cal V}(t) \cap X.$$  
  Then
$V_i \in {\cal C}_i$, so that  ${\cal C}_i$ is a nonempty patch closed subset of $X$.    

  We use compactness to show that $\bigcap_{i>0}{\cal C}_i$ is nonempty.  To this end, we claim that the collection $\{{\cal C}_i:i>0\}$ has the finite intersection property. 
 %
Let $i>0$, and let $0 < j < i$. Then $$V_i \in {\cal C}_j =  {\cal U}(mt^j) \cap {\cal V}(t) \cap X$$
 since $mt^{i} \in V_{i}$ and $t^{-1} \in V_{i}$ implies $mt^j = mt^{i}(t^{-1})^{i-j} \in V_{i}$. 
 For each $i>0$, it follows that ${\cal C}_1 \cap {\cal C}_2 \cap \cdots \cap {\cal C}_i$ contains the valuation ring $V_{i}$. 
Therefore, the collection  $\{ {\cal C}_i:i>0\}$  of patch closed subsets of $X$
 has the finite intersection property. Since $X$ is a patch closed subset of the patch quasicompact space $\X(F)$, $X$ is patch quasicompact. Thus the set $$ \bigcap_{i>0} {\cal C}_i \: =\:  X  \cap {\cal V}(t) \cap  \left(\bigcap_{i>0}{\cal U}(mt^i)\right)$$ is nonempty. Let $U$ be a valuation ring in this intersection. Then $U \in X$, and, for each $i>0$, we have $0 \ne m \in (t^{-1})^iU$. Also, since $t \not \in U$, we have $t^{-1} \in {\ff M}_U$. Thus $m \in (t^{-1})^iU \subseteq {\ff M}_U$.  
   If $U$ has rank $1$, then the radical of $mU$ in $U$ is ${\ff M}_U$.  In this case there exists $n>0$ such that $$(t^{-1})^n \in mU \subseteq (t^{-1})^{n+1}U,$$ a contradiction to the fact that $t^{-1}$ is in $U$ but is not a unit in $U$. We conclude that $U$ has rank $>1$.   
\end{proof}

In the proof of Theorem~\ref{ugh} we pass to a subfield $K$ of $F$. In doing so, we need that features of the topology of $X$ are preserved in the image of $X$ in the Zariski-Riemann space of $K$. This is given by the next lemma.

\begin{lemma} \label{closed map} Let $K$ be a subfield of $F$. Then the mapping $f:\X(F) \rightarrow \X(K):V \mapsto V \cap K$ is a surjective map that is closed and continuous in the patch topology.
\end{lemma}

\begin{proof}
 That $f$ is surjective follows from the Chevalley Extension Theorem \cite[Theorem 3.1.1, p.~57]{EP}. 
To see that $f$ is continuous in the patch topology observe that 
for $x_1,\ldots,x_n,y \in K$,  
the preimages under $f$ of the subbasic patch open sets 
\begin{center} $\{V \in \X(K):x_1,\ldots,x_n \in V\}$ \: and \: $\{V \in \X(K):y \not \in V\}$ \end{center} are ${\cal U}(x_1,\ldots,x_n)$ and ${\cal V}(y)$, respectively, and hence are patch open in $\X(F)$.  
%
Finally,  with the patch topologies on $\X(F)$ and $\X(K)$,  $f$ is a continuous map between compact Hausdorff spaces, and so $f$ is closed \cite[Theorem~3.1.12, p.~125]{Eng}.
\end{proof}

\begin{theorem} \label{ugh} Let $ X$ be a nonempty subset of  $\X(F)$ such that $J(X) \ne 0$. If  $A(X)$ is a local ring that is not a valuation domain, then 
 $\patch(X)$ contains a valuation ring of rank $>1$. 
\end{theorem}

\begin{proof} Let $A = A(X)$, $J = J(X)$, and let 
 $M$ denote the unique maximal ideal of $A$.  
We prove the lemma by establishing  a series of  claims.  In the proof, for an ideal $I$ of $A$, we denote by $\End(I)$ the subring of $F$ given by $\{t \in F:tI \subseteq I\}$. Thus $\End(I)$ is the largest ring in $F$ in which $I$ is an ideal.  

\medskip

{\textsc{Claim 1.}} If $A$ is not completely integrally closed\footnote{A domain $R$ is completely integrally closed if for each nonzero ideal $I$ of $R$ and each element $t$ of the quotient field of $R$, $tI \subseteq I$ if and only if $t \in R$; equivalently, $R = \End(I)$ for each nonzero ideal $I$ of $R$.}, then $\patch(X)$ contains a valuation ring of rank $>1$.   

\medskip 

{\it Proof of Claim 1.} If every valuation ring in $\patch(X)$ has rank $\leq 1$, then $A$, as an intersection of completely integrally closed domains, is completely integrally closed. \qed

\medskip

{\textsc{Claim 2.}} If $M$ is a principal ideal of $A$, then $\patch(X)$ contains a valuation ring of rank~$>1$.   

\medskip

{\it Proof of Claim 2.} 
 Suppose  $M = mA$ for some $m \in M$. Since
 $A$ is not a valuation domain, $A$ is not field, and hence $M \ne 0$.   
  Since $M$ is principal, the ideal  $P = \bigcap_{i>0}M^i$ is the unique largest nonmaximal prime ideal of $A$ and $PA_P  = P$ \cite[Exercise 1.5, p. 7]{Kap}. If $P = 0$, then $A$ is valuation ring (in fact, a DVR), contrary to assumption.  
 Thus $P \ne 0$, and, since $P = PA_P$, it follows that $A_P \subseteq \End(P)$.  
 Since $P$ is a nonmaximal prime ideal, this implies $A \subsetneq \End(P)$, so that $A$ is not completely integrally closed. By Claim 1, 
  $\patch(X)$ contains a valuation ring of rank $>1$, which proves Claim~2.   \qed

\medskip

{\textsc{Claim 3.}} If all the valuation rings  in $X$ dominate $A$, then $\patch(X)$ contains a valuation ring of rank $>1$.    

\medskip 

{\it Proof of Claim 3.}
If $M$ is a principal ideal of $A$,   Claim 2 implies that $\patch(X)$ contains a valuation ring of rank $>1$, and the proof of the claim is complete. Thus we assume $M$ is not a principal ideal of $A$. Also, by Claim 1, we may assume that $A$ is completely integrally closed and hence $\End(M) = A$.  
It remains to prove Claim~3 in the case in which $\End(M) = A$, $M$ is not a principal ideal of $A$ and all the  valuation rings  in $X$ dominate $A$.

 Since $A$ is not a valuation ring, there exists $t \in F$ such that $t \not \in A$ and $t^{-1} \not \in A$.  
Since $M$ is not an invertible ideal of $A$, we have $$M \subseteq M(A:_FM) \subsetneq A,$$ which forces $M = M(A:_FM)$. Thus $$(A:_FM) = \End(M)  =A,$$ so  $t \not \in A = (A:_FM)$. Let $m \in M$ such that $mt \not \in A$. Since $t^{-1},mt \not \in A$, we have $mt^i, (mt^i)^{-1} \not \in A$ for each $i>0$, as noted at the beginning of the proof of Lemma~\ref{step}.  
By Theorem~\ref{trans} (applied in the case where ``$D$'' is $A$), there exists, for each $i>0$, a valuation ring $V_i \in \patch(X)$ that dominates $A$ and for which $mt^i$ is a unit in $V_i$.
By assumption, every valuation ring in $X$ dominates $A$, so Lemma~\ref{list}(1) implies that every valuation ring in $\patch(X)$ dominates $A$. Thus $m \in {\ff M}_{V_i}$ for all $i$. 
  By Lemma~\ref{step}, $\patch(X)$ contains a valuation ring of rank $>1$, which completes the proof of Claim 3. \qed

\medskip

{\textsc{Claim}} 4.  {If  there are valuation rings in $X$ that do not dominate $A$, then there exists a two-dimensional Noetherian local subring $D$ of $A$ such that $D$ is dominated by $A$ and $D$ contains nonzero elements $a,b$ with $a \in M \setminus J$ and $b \in J$. }

\medskip

{\it Proof of Claim 4.}
Since  there are valuation rings in $X$ that do not dominate $A$, we have $0 \ne J \subsetneq M$.  
To prove the existence of the ring $D$, suppose first  $A$ contains a field $k$.  
In this case, choose $a \in M \setminus J$ and $0 \ne b \in J$. Let $C = k[a,b]$ and let $D = C_{M \cap C}$. Then $D$ is a Noetherian local subring of $A$ that is dominated by $A$. Since $C$ is generated by two elements over a field, $D$ has Krull dimension at most two. To see that $D$ has exactly Krull dimension two, observe that  
 since $a \in M \setminus J$ there is $V \in X$ such that $a \not \in {\ff M}_V$. Thus $0 \ne b \in {\ff M}_V \cap D \subsetneq M \cap D$, so that $D$  has Krull dimension two. 
 This verifies Claim 4 if $A$ contains a field.

Suppose next that  $A$ does not contain a field.   
Then the contraction of $M$ to the prime subring of $A$  is a nonzero principal ideal, and hence there is a DVR $W$ that is dominated by $A$. Let $p$ be the generator of the maximal ideal of $W$. 

If $p \in J$, then let $b = p$ and choose $a \in M \setminus J$. 

If $p \not \in J$, then let $a = p$ and choose $0 \ne b \in J$. 

In either case, we have $a \in M \setminus J$ and $0 \ne b \in J$.   Let $C = W[a,b]$, and let $D = C_{M \cap C}$.  Since $W$ is a DVR and either $a \in W$ or $b \in W$, the ring $D$ has Krull dimension at most two. As in the case in which $A$ contains a field, since $a \in M \setminus J$ and $J \ne 0$, it follows that $D$ has Krull dimension two. This verifies Claim 4. 
\qed



\medskip

{\textsc{Claim 5.}} The set $\patch(X)$ contains a valuation ring of rank $>1$.    

\medskip 

{\it Proof of Claim 5.}
If every valuation ring in $X$ dominates $A$, then, by Claim 3, $\patch(X)$ contains a valuation ring of rank $>1$. It remains to consider the case where there are valuation rings in $X$ that do not dominate $A$.  
By Claim 4, there is 
a two-dimensional local Noetherian subring $D$ of $A$ such that $D$ is dominated by $A$ and $D$ contains nonzero elements $a,b$ such that $a \in M \setminus J$ and $b \in J$.  The next step of the proof involves a reduction that allows us to work in the quotient field of $D$.

Let $K$ denote the quotient field of $D$. Let
\begin{center}
 $A' = A \cap K$,
\:\: $M' = M \cap K$,  \:\: and \:\:  
 $X' = \{V \cap K:V \in \patch(X)\}.$
 \end{center}
    Then $A'$ is a local integrally closed overring\footnote{By an {\it overring} of a domain, we mean a ring between the domain and its quotient field.} 
 of $D$ with maximal ideal $M'$ and  $A' = \bigcap_{V \in X'}V$. Also, by Lemma~\ref{closed map}, $X'$ is patch closed in $\Zar(K)$.  

\medskip

{\textsc{Claim 5}}(a). All but finitely many valuation rings in $X'$ dominate $D$.     

\medskip

 {\it Proof of Claim 5(a)}. Suppose that $V \in X'$ and $V$ does not dominate $D$. Then ${\ff M}_V \cap D$ is a nonmaximal prime ideal of $D$ that by Lemma~\ref{list}(1) contains $b$.  Since $D$ is a Noetherian ring of Krull dimension $2$, there are only finitely many  height one prime ideals $P_1,\ldots,P_n$  of $D$ containing $b$. The valuation ring $V$ contains the integral closure $D'$ of the semilocal one-dimensional Noetherian ring $D_{P_1} \cap \cdots \cap D_{P_n}$. Since the ring $D'$ is a semilocal PID,  there are  only finitely many valuation rings between $D'$ and its quotient field $K$. Therefore, the set of valuation rings in $X'$ that do not dominate $D$ is finite. \qed
 
 \medskip

Returning to the proof of Claim 5, to show  that $\patch(X)$ contains a valuation ring of rank~$>1$, it is enough to prove that $X'$ contains a valuation ring  of rank $>1$. This is because
if $U \in X'$ has rank $>1$, there is a valuation ring $V \in \patch(X)$ with $V \cap K = U$, and $V$ is necessarily of rank $>1$ since $V $ extends $ U$. 

Thus we need only show that $X'$ contains a valuation ring of rank $>1$.  We prove this by verifying two claims.

\medskip

{\textsc{Claim 5}}(b). If $A'[1/a]$ is a valuation ring, then $X'$ contains a valuation ring of rank $>1$.   

\medskip

{\it Proof of Claim 5(b)}. Since $a \not \in J$, there exists $V \in X'$  such that $a \not \in {\ff M}_V$, and hence $1/a \in V$. Since $0 \ne b \in  J \cap A' \subseteq {\ff M}_V$, the valuation ring $V$ does not have rank~$0$.  
Therefore, 
$A'[1/a] \subseteq V \subsetneq K$. As an overring of the two-dimensional Noetherian ring $D$, $A'$ has Krull dimension at most two. (This follows from the Dimension Inequality \cite[Theorem 15.5, p.~118]{Mat}.) Since $a \in M'$, the Krull dimension of $A'[1/a]$ is less than that of $A'$, and hence, since $A'[1/a] \subsetneq K$, it follows that $A'[1/a]$ has Krull dimension  one (and $A'$ has Krull dimension two). Thus the valuation ring $A'[1/a]$ has  rank one.  



Now suppose by way of contradiction that every valuation ring in $X'$ has rank $\leq 1$.  Then $A'$, as an intersection of valuation rings in $X'$, is completely integrally closed. Thus, 
%
since 
 $A'$ has Krull dimension $2$,  $A'$ is not a valuation  domain. 
 
 By \cite[(4.2)]{OGraz} every  patch closed representation of a domain contains a minimal patch closed representation. Thus there  is a patch closed subset $Y$ of $X'$ (recall that $X'$ is patch closed) such that $A' = \bigcap_{V \in Y}V$ and no proper patch closed subset of $Y$ is a representation of $A'$.  
 By Lemma~\ref{list}(4), the rank one valuation ring $A'[1/a]$, as  a proper subring of $K$, is an intersection of valuation rings in $Y$. 
 Since $A'[1/a]$ has the same quotient field as the valuation rings in $X'$,    
 we conclude that 
$A'[1/a]$ is  in $Y$. 

Since $Y \subseteq X'$ and $J \cap A' = \bigcap_{V \in X'}{\ff M}_{V}$, we have by Lemma~\ref{list}(1) that $0 \ne b \in J \cap A' \subseteq {\ff M}_V$ for all $V \in Y$. In particular, $K \not \in Y$.   Let $W = A'[1/a]$. 
  Since $W$ has rank one, we conclude that  $\{W\} = {\cal U}(1/a) \cap Y$, so that $W$ is a patch isolated point in $Y$. We show this leads to a contradiction to the fact that $Y$ is a minimal patch closed representation of $A'$.  
  
  Since $W$ is a patch isolated point in $Y$, $Y \setminus \{W\}$ is a patch closed set, and hence by the minimality of $Y$ we have have $$A' \:\: \subsetneq \: \bigcap_{U \in Y  \setminus \{W\}}U.$$ Let  $$I \: \: =\bigcap_{U \in Y  \setminus \{W\}}{\ff M}_U.$$ Then $$\bigcap_{U \in Y  \setminus \{W\}}U \: \subseteq \: \: \End(I).$$ Since $A'$ is completely integrally closed, it follows that $I = 0$ or $I \not \subseteq A'$. The former case is impossible, since $0 \ne b \in J \cap A' \subseteq I$. Thus we conclude that $I \not \subseteq A'$.  Let $t \in I \setminus A'$.  Then $t^{-1} \not \in A'$ since for any $U \in Y \setminus \{W\}$, we have $t \in {\ff M}_U$ and $A' \subseteq U$.  Thus $t,t^{-1} \not \in A'$.  
 
 By Claim 5(a), all but finitely many valuation rings in $X'$, hence also in $Y$, dominate $D$. Applying Theorem~\ref{trans} to $D$, $A'$, $Y$  and $t$,  there  
exists a valuation ring $U \in \lim(Y)$ such that $t,t^{-1} \in U$.   
Since $Y$ is patch closed, $U \in Y$, and, since $t$ is a unit in $U$, $t \not \in {\ff M}_U$.  By the choice of $t$, $t$ is in the maximal ideal of every valuation ring in $Y$ except $W$. This forces $W = U \in \lim(Y)$, contradicting the fact that $W$ is a patch isolated point in $Y$. This contradiction shows that  
  $X'$ contains a valuation ring of rank $>1$.  This 
  completes the proof of Claim 5(b).  \qed
  
\medskip

{\textsc{Claim 5}}(c). If $A'[1/a]$ is not a valuation ring, then $X'$ contains a valuation ring of rank~$>1$.   

\medskip

{\it Proof of Claim 5(c)}. 
Since $A'[1/a]$ is not a valuation ring and $A'$ has quotient field $K$, there exists $0 \ne t \in K$ such that $t \not \in A'[1/a]$ and $t^{-1} \not \in A'[1/a]$.  Thus $t,t^{-1} \not \in A'$ and $at \not \in A'$. With the aim of applying Lemma~\ref{step} to $A'$ and $X'$, we fix $i>0$ and we show that there is a valuation ring $V \in X'$ such that $at^i$ is a unit in $V$ and $a \in {\ff M}_V$. Once this is proved, Lemma~\ref{step} implies that $X'$ contains a valuation ring of rank $>1$, and the proof of Claim 5(c) is complete. 

Let $s = at^i$.  Since $t \not \in A'[1/a]$ and $A'[1/a]$ is integrally closed, it follows that $t^i \not \in A'[1/a]$, and hence $s \not \in A'$. If $s^{-1} \in A'$, then $t^{-i} = a(a^{-1}t^{-i}) =as^{-1} \in A'$, so that, since $A'$ is integrally closed in $K$, $t^{-1} \in A'$, contrary to the choice of $t$. Therefore, $s,s^{-1} \not \in A'$.  
 By Claim 5(a), all but finitely many valuation rings in $X'$ dominate $D$. By Theorem~\ref{trans}, 
 with  $D, A', X'$ and $s$ playing 
 the roles of ``$D$'', ``$A$'', ``$X$'' and ``$t$'' in the theorem, 
  we obtain    $U \in X'$ such that $s,1/s \in U$ and $U$ dominates $D$.  Therefore, $s=at^i$ is a unit in $U$ and $a \in {\ff M}_U$ since $U$ dominates $D$. By Lemma~\ref{step}, $X'$ contains a valuation of rank $>1$, which proves Claim 5(c).      \qed
     
     \medskip
     
     Finally, to complete the proof of Claim 5, we note that Claim 5(b) and 5(c) show that $X'$ contains a valuation ring of rank $>1$.  As discussed after the proof of Claim 5(a), this implies that $\patch(X)$ contains a valuation of rank $>1$. Therefore, with the proof of Claim 5 complete, the proof of the theorem is complete also. 
%
     \end{proof}
     
     \section{Compact sets and holomorphy rings} 
     
          The main results of the paper involve one-dimensional Pr\"ufer domains. We collect in the next lemma some basic properties of such rings that are needed for the theorems in this section. We denote by $J(A)$ the Jacobson radical of a ring $A$, and by $\Max(A)$ the space of maximal ideals of $A$ endowed with the Zariski topology.  
         
\begin{lemma} \label{A lemma} Let  $A$ be a one-dimensional Pr\"ufer domain with quotient field $F$, and let 
$X \subseteq \X(F)$  such that $F \not \in X$ and   $A = A(X)$. Then $J(A) = J(X)$.  If 
 $J(A) \ne 0$, then the Zariski, patch and inverse topologies all coincide on $X$ and  $$\patch(X) = \{A_M:M \in \Max(A)\}.$$      

\end{lemma}

\begin{proof}  
It is straightforward to check that  $J(X) \subseteq J(A)$; for example, see \cite[Remark~1.3]{HNoeth}. To see that the reverse inclusion holds, let $V \in X$. Since $F \not \in X$ and $A$ is a one-dimensional Pr\"ufer domain,  $V = A_M$, where $M = {\ff M}_V \cap A$ is a maximal ideal of $A$.  Since $J(A) \subseteq M \subseteq {\ff M}_V$, we have $J(A) \subseteq   J(X)$. This proves the first assertion of the lemma. 

Suppose now that $J(A) \ne 0$.  Since $A$ has Krull dimension one,  $\Max(A)$ is homeomorphic to the spectral space $\Spec(A/J(A))$. In a spectral space for which every point is both minimal and maximal with respect to the specialization order, the Zariski, patch and inverse topologies all agree; cf.~\cite[Corollary 2.6]{ST} or use the fact that $A/J(A)$ is a von Neumann regular ring. Thus these three topologies all agree on $X$ since $X$ is homeomorphic to a subspace of $\Max(A)$.    

Finally,  since $A$ is a one-dimensional Pr\"ufer domain represented by $X$, the
 set of all valuation overrings of $A$ (each of which must have rank $\leq 1$) is $\patch(X) \cup \{F\}$  
  \cite[Corollary 4.10]{FFL}.  Since $J(A) \ne 0$, Lemma~\ref{list}(1) implies $F \not \in \patch(X)$. 
  Therefore, $\patch(X) = \{A_{M}:M \in \Max(A)\}$.  
\end{proof}

\begin{remark}
Topological aspects and factorization theory of one-dimensional Pr\"ufer domains with nonzero Jacobson radical are studied in \cite{HOR}.  
\end{remark}

                    The first application of the results of the previous section is the following characterization of subsets of $\X(F)$ whose holomorphy ring is a one-dimensional Pr\"ufer domain with nonzero Jacobson radical and quotient field $F$.  
                    
      \begin{theorem} \label{rank one theorem} The following are equivalent for a nonempty subset $X$ of  $ \X(F)$  with $J(X) \ne 0$.    
       
      \begin{enumerate}[$(12)$]
      \item[{\em (1)}] $A(X)$  is a one-dimensional Pr\"ufer domain with   quotient field $F$. 
      \item[{\em (2)}] 
       $X$ is contained in a    quasicompact set of rank one valuation rings in $\X(F)$.
       
       \item[{\em (3)}]  Every valuation ring in $\patch(X)$ has rank one. 
       \end{enumerate}
  \end{theorem}
  
  \begin{proof}
  Let $A = A(X)$ and $J = J(X)$. 
  
 (1) $\Rightarrow$ (2)  Since $A$ is a one-dimensional  Pr\"ufer domain with quotient field $F$, the subset of $\X(F)$ given by  $Y = \{A_M:M \in \Max(A)\}$ consists of rank one valuation rings in $\X(F)$. The only other valuation overring of $A$ is $F$. Since $J \ne 0$, we have $F \not \in X$, which forces 
  $X \subseteq Y$.  Moreover, $Y$ is homeomorphic to $\Max(A)$ and the maximal spectrum of a ring  is quasicompact, so  statement (2) follows.

  (2) $\Rightarrow$ (3)  Let $Y$ be a    quasicompact set of rank one valuation rings in $\X(F)$ such that $X \subseteq Y$. Since $Y$ is quasicompact, Lemma~\ref{list}(3) implies  $\patch(X) \subseteq Y \cup \{F\}$.
   If $F \in \patch(X)$, then 
 from Lemma~\ref{list}(1) it follows that $J =0$, contrary to assumption. Therefore, $\patch(X) \subseteq Y$, so that $\patch(X)$ consists of rank one valuation rings. 
 
 (3) $\Rightarrow$ (1) 
 By Lemma~\ref{list}(1), $A = \bigcap_{V \in \patch(X)}V$ and  $0 \ne J = \bigcap_{V \in \patch(X)}{\ff M}_V$.  
 Thus we can assume without loss of generality that $X = \patch(X)$.  
 We claim first that $A$ has quotient field $F$.  
 Let $0 \ne a \in J$. By Lemma~\ref{list}(2), $X$ is     quasicompact, so by Lemma~\ref{list}(4) we have $$A[1/a] \: \: =  \: \bigcap_{1/a \in V \in X \cup \{F\}}V \: = \: F,$$ where the last equality follows from the fact that every valuation ring $V$ in $X$ has rank one and satisfies  $a \in {\ff M}_V$. Since $A[1/a] = F$, we conclude that $A$ has quotient field $F$.

  
   To prove that $A$ is  a one-dimensional Pr\"ufer domain, it suffices to show that $A_M$ is a rank one valuation domain for each maximal ideal $M$ of $A$. 
  Let $M$ be a maximal ideal of $A$.  Let $Y = \{V \in X:A_M \subseteq V\}$. Since $$Y = X \cap (\bigcap_{t \in A_M}{\cal U}(t)),$$ $Y$ is  patch closed in $X$. By Lemma~\ref{list}(2) and the fact that $X$ is patch closed in $\Zar(F)$, $Y$ is a quasicompact subset of $\X(F)$.  Since $Y$ is quasicompact and consists of rank one valuation rings,   Lemma~\ref{list}(4) implies that   $A_M= \bigcap_{ V \in Y}V$.
Thus $Y$ is  a patch closed representation of $A_M$ consisting of rank one valuation rings. Since $J \ne 0$, Theorem~\ref{ugh} implies that $A_M$ is a valuation domain. Since $A_M$ has quotient field $F$ and $Y$ is a representation of $A_M$ consisting of rank one valuation domains, it follows that $A_M \in Y$. Hence $A_M$ is  a rank one valuation domain, which proves that  
%
 $A$ is a Pr\"ufer domain with  Krull dimension one.  
  \end{proof}

A domain $A$ is an {\it almost Dedekind domain} if for each maximal ideal $M$ of $A$, $A_M$ is a DVR.  There exist many interesting examples of almost Dedekind domains; see for example \cite{HOR, Loper, OFact} and their references. For the factorization theory of such rings, see \cite{FHL, HOR, LL}.


      \begin{corollary} \label{ADD cor}  Let $X$ be a nonempty set of $ \X(F)$ such that $J(X) \ne 0$. Then $X$     
 is contained in a    quasicompact set of DVRs in $\X(F)$ if and only if 
      $A(X)$  is an almost Dedekind  domain with quotient field $F$.  
    
  \end{corollary}
  
  \begin{proof}  Let $A = A(X)$.  
  Suppose $X$ is contained in a    quasicompact set $Y$ of DVRs in $\X(F)$.  By Theorem~\ref{rank one theorem}, $A$ is a one-dimensional Pr\"ufer domain with quotient field $F$ and nonzero Jacobson radical. 
  By Lemma~\ref{list}(3), $$\patch(X) \subseteq \patch(Y) = Y \cup \{F\}.$$ Since $J(X) \ne 0,$  Lemma~\ref{list}(1) implies $F \not \in \patch(X)$, so $\patch(X) \subseteq Y$. Thus $\patch(X)$ consists of DVRs.  By Lemma~\ref{A lemma}, we have that for 
     each maximal ideal $M$ of $A$, $A_M$ is in $\patch(X)$ and hence $A_M$ is a DVR.  Thus $A$ is an almost Dedekind domain.  The converse follows from 
  Lemma~\ref{A lemma} and Theorem~\ref{rank one theorem}. 
  \end{proof} 
  
  \begin{remark}{Let $X$ be a nonempty quasicompact set of DVRs in $ \X(F)$ such that $J(X) \ne 0$. 
By Corollary~\ref{ADD cor}, $A = A(X)$ is an almost Dedekind domain, and  $J(A) = J(X)$ by Lemma~\ref{A lemma}.  If there is $t \in J(A)$ such that $J(A) = tA$ (equivalently, ${\ff M}_V = tV$ for each $V \in X)$, then 
 $A$ has the property that every proper ideal is a product of radical ideals. For this and related results on such rings, which are known  in the literature as SP-domains or domains  with the radical factorization property, see \cite{FHL, HOR, OFact}. 
 }
\end{remark}

We prove next our main theorem of this section (the ``main lemma'' of the introduction) regarding the correspondence between quasicompact sets and holomorphy rings in the space of rank one valuation rings.

\begin{theorem} \label{correspond 1} The mappings   \begin{center} $X \mapsto A(X)$  \: and \: $A \mapsto \{A_M:M \in \Max(A)\}$ \end{center}
 define a bijection between
  the quasicompact sets  $X$ of rank one valuation rings in $\Zar(F)$ with $J(X) \ne 0$
 and
the one-dimensional Pr\"ufer  domains $A$ with quotient field $F$ and  nonzero Jacobson radical.
\end{theorem}

\begin{proof}  Let $X$ be a 
quasicompact set of rank one valuation rings in $\Zar(F)$ with $J(X) \ne 0$.  
 By   Lemma~\ref{A lemma} and Theorem~\ref{rank one theorem}, $A = A(X)$ is a one-dimensional Pr\"ufer domain such that  $J(A) \ne 0$ and $A$ has quotient field $F$.  By Lemma~\ref{list}(3), $\patch(X) \subseteq X \cup \{F\}$. Since $J(A) \ne 0$, Lemma~\ref{list}(1) implies $F \not \in \patch(X)$.  Thus $X = \patch(X)$, and, by 
Lemma~\ref{A lemma}, 
 $X 
 = \{A_M:M \in \Max(A)\}$. 
Conversely, suppose $A$ is one-dimensional Pr\"ufer overring with quotient field $F$ and $0 \ne t \in J(A)$. Since $\Max(A)$ is quasicompact,  $X=\{A_M:M \in \Max(A)\}$ is  a  quasicompact set of rank one valuation rings  such that $A = A(X)$ and $0\ne t \in J(X)$.  
 \end{proof}

 The next example
shows the necessity of the hypotheses in Theorem~\ref{correspond 1}.  
 
 \begin{example} \label{main example} 
 Let $X$ be a nonempty subset of $\Zar(F)$. Theorem~\ref{correspond 1} shows that if (a) $X$ consists of rank one valuation rings, (b) $X$ is quasicompact and (c) $J(X) \ne 0$, then $A(X)$ is a Pr\"ufer domain. 
 These hypotheses  are necessary in the sense that $A(X)$ need not be a Pr\"ufer domain if any one of (a), (b) or (c) is omitted.  The following classes of examples illustrate this. 
\begin{enumerate}[$(1)$]
\item {\it An example in which $A(X)$ is not a Pr\"ufer domain but in which (a) and (b) hold.} Let $D$ be an integrally closed  Noetherian domain of Krull dimension $>1$. Then the set $X$ of localizations of $D$ at height one prime ideals is a quasicompact set of rank one valuation rings for which $A(X) = D$. 

\item {\it An example in which $A(X)$ is not a Pr\"ufer domain but in which (b) and (c) hold.}  Let $k$ be a field, and let $V$ be the DVR $k(S)[T]_{(T)}$, where $S$ and $T$ are indeterminates for $k$.  Let $R = k + {\ff M}_V$.  Then $R$ is a one-dimensional integrally closed local domain. With $X$ the set of valuation overrings of $R$ of rank $>0$, we have that $R = A(X)$ and  $J(X) \ne 0$. Moreover,  $X$ is quasicompact since $X$ is the closed subset of the (quasicompact) space of valuation overrings. Indeed, $X$ is the set of valuation overrings of $R$ contained   in $V$.  
%
Thus (b) and (c) hold, but $A(X)$ is not a Pr\"ufer domain since $A(X)$ is a local domain contained in more than one valuation ring of Krull dimension $2$. 

\item {\it An example in which $A(X)$ is not a Pr\"ufer domain but in which (a) and (c) hold.} Let $D$ be an integrally closed Noetherian local domain of Krull dimension~$>1$ with quotient field $F$.  To exhibit the desired example, it suffices to show  that $D$ is the intersection of the DVRs in $\Zar(F/D)$ that dominate $D$.  Let ${\ff m}$ denote the maximal ideal of $D$, and let $x \in F \setminus D$. If $x^{-1} \in D$, then $x^{-1} \in {\ff m}$. Since every Noetherian local domain is birationally dominated by a DVR 
\cite[p.~26]{Ch}, there exists a DVR $V$ in $\Zar(F/D)$  such that $x^{-1} \in {\ff M}_V$. Thus $x \not \in V$.  On the other hand, if $x^{-1} \not \in D$, then the ring $D[x^{-1}]$ has a maximal ideal generated by ${\ff m}$ and $x^{-1}$ \cite[Exercise 31, pp.~43--44]{Kap}, so, again by \cite[p.~26]{Ch}, there is a DVR $V$ in $\Zar(F/D)$ that dominates $D$ and for which $x^{-1} \in {\ff M}_V$ and hence $x \not \in V$. It follows that $D$ is the intersection of all the DVRs in $\Zar(F/D)$ that dominate it. 

\end{enumerate}

\end{example} 
 
 Restricting to DVRs in Theorem~\ref{correspond 1} yields a correspondence with almost Dedekind domains.
 
 \begin{corollary} \label{correspond 1 cor} 
 The mappings  \begin{center} $X \mapsto A(X)$  \: and \:   $A \mapsto \{A_M:M \in \Max(A)\}$ \end{center}
 define a bijection between
  the quasicompact sets  $X$ of DVRs in $\Zar(F)$ with $J(X) \ne 0$
 and
the almost Dedekind domains $A$ with quotient field $F$ and nonzero Jacobson radical.  
\end{corollary}

\begin{proof} If $X$ is a 
quasicompact set of DVRs in $\X(F)$ with $J(X) \ne 0$, then, by Corollary~\ref{ADD cor}, $A$ is an almost Dedekind domain. 
Conversely, if  $A$ is almost Dedekind domain, then clearly
  $X=\{A_M:M \in \Max(A)\}$ is  a  quasicompact set of  DVRs. Thus the corollary follows  from Theorem~\ref{correspond 1}.  
 \end{proof}

By Proposition~\ref{T2}(2) the quasicompact sets in Theorem~\ref{correspond 1} are compact, so the theorem alternatively can be stated for compact sets instead. 
Along these lines, in the case in which all the valuation rings under consideration occur as overrings of a domain with quotient field $F$, the restriction that $J(X) \ne 0$ in Theorem~\ref{correspond 1} can be omitted (or, more correctly, hidden) if  the quasicompact hypothesis is strengthened to that of being compact. 

\begin{corollary} \label{correspond 2} Let $R$ be a proper subring of $F$ with quotient field $F$. There is a bijection given by  \begin{center} $X \mapsto A(X)$  \: and \:  $A \mapsto \{A_M:M \in \Max(A)\}$ \end{center}
 between the compact sets  $X$ of rank one valuation overrings of $R$ and the one-dimensional Pr\"ufer  overrings $A$ of $R$  with   nonzero Jacobson radical.   
\end{corollary}

\begin{proof}  
If $X$ is a    compact set of rank one valuation overrings, then $J(X) \ne 0$ by Proposition~\ref{T2}(3). By Theorem~\ref{correspond 1}, 
 $A = A(X)$ is a one-dimensional Pr\"ufer domain with $J(A) \ne 0$ and  $X 
 = \{A_M:M \in \Max(A)\}$. 
Conversely, if $A$ is one-dimensional Pr\"ufer overring of $R$ with $J(A) \ne 0$, then 
 $X=\{A_M:M \in \Max(A)\}$ is    quasicompact by Theorem~\ref{correspond 1}, and $X$ is  Hausdorff by Proposition~\ref{T2}(2). 
Clearly, $A = A(X)$.  
\end{proof}

\begin{remark} 
As in Corollary~\ref{correspond 1 cor}, the bijection in Corollary~\ref{correspond 2} restricts to a bijection between compact sets  of DVRs   and almost Dedekind overrings $A$ of $R$ with nonzero Jacobson radical. 
\end{remark} 

In general it is difficult to determine when an intersection of two one-dimensional Pr\"ufer domains with quotient field $F$ is a Pr\"ufer domain. For example, it is easy to see any Noetherian local UFD $A$ of Krull dimension $2$ can be written as an intersection $A = A_1 \cap A_2$ where  $A_1$ is a DVR overring and $A_2$ is  a PID overring. In this example, $A$ is not a Pr\"ufer domain despite being an intersection of two PIDs. Significantly, the ring $A_2$ here has $J(A_2) = 0$.  
 In our context, the topological characterization in Corollary~\ref{correspond 2} shows that the difficulty here is removed if $J(A_i) \ne 0$, a fact we prove in the next corollary.

  \begin{corollary} \label{underring} Let $A_1,\ldots,A_n$ be one-dimensional Pr\"ufer domains, each  with quotient field $F$.  Let $J =  J(A_1) \cap \cdots \cap J(A_n)$, and 
let  $A = A_1 \cap \cdots \cap A_n$. If $J(A_i) \cap A \ne 0$ for all $i$,   
  then $A$ is a one-dimensional Pr\"ufer domain with $J(A) = J$ and quotient field $F$. If also each $A_i$ is an almost Dedekind domain, then so is $A$.   
\end{corollary} 

\begin{proof} 
For each $i=1,2,\ldots,n$, let $$X_i =\{(A_i)_M:M \in \Max(A_i)\}.$$  By  Corollary~\ref{correspond 2}, each $X_i$ is compact.  Thus $X = X_1 \cup \cdots \cup X_n$ is quasicompact. Since $J(A_i) \cap A \ne 0$ for all $i$, we have $$J(X)  = J(A_1) \cap \cdots \cap J(A_n) \ne 0.$$ 
Thus Theorem~\ref{correspond 1} implies 
 $A = A(X)$ is a one-dimensional Pr\"ufer domain, and, by Lemma~\ref{A lemma}, $J(A) = J(A_1) \cap \cdots \cap J(A_n)$. 
%
If also $A_i$ is an almost Dedekind domain, then each $X_i$ is a   quasicompact set of DVRs, so that $X$ is also a   quasicompact set of DVRs.  In this case,  $A$ is an almost Dedekind domain by Corollary~\ref{correspond 1 cor}. 
\end{proof}

\begin{remark}  It is an open question as to whether the intersection $A = A_1 \cap A_2$ of one-dimensional Pr\"ufer domains $A_1$ and $A_2$ with quotient field $F$ and nonzero Jacobson radical   has quotient field $F$. 
 If the answer is affirmative, then $A$ is a one-dimensional Pr\"ufer domain by Corollary~\ref{underring}. In any case, let $D$ be the prime subring of $A$, let $S_i = \{1 + d:d \in J(A_i)\}$ and let $B_i = D_{S_i} + J(A_i)$. Then  $B_1$ and $B_2$ are local domains of Krull dimension one with quotient field $F$.  A necessary condition for $A$ to have quotient field $F$ is that $B_1 \cap B_2$ has quotient field $F$.  In \cite[Question~2.1]{GH}  Gilmer and Heinzer  ask whether the intersection of two one-dimensional local domains, each with the same quotient field $F$, has quotient field $F$? 
\end{remark}

\begin{corollary} \label{agree} Let $X$ be a nonempty quasicompact set of rank one valuation rings in $\Zar(F)$. If  $J(X) \ne 0$, then the patch, inverse and Zariski  topologies agree on $X$.
\end{corollary} 

\begin{proof}
By 
 Theorem~\ref{correspond 1},   $A=A(X)$ is a one-dimensional Pr\"ufer domain with $J(A) \ne 0$ and quotient field $F$. By Lemma~\ref{A lemma}, the  inverse, patch and Zariski topologies agree on $X$.
\end{proof} 

As the last application of this section, we describe  schemes  in $\X(F)$ consisting of rank $\leq 1$ valuation rings. 
 Let   $X$ be a subspace of $\X(F)$, and let $X^* =X \cup \{F\}$. 
 Let $\O_{X^*}$ be the sheaf on $X^*$ defined for each nonempty open set ${\cal U}$ of ${X^*}$ by $\O_{X^*}({\cal U}) = \bigcap_{V \in {\cal U}}V$.  (The reason for appending $F$ to $X$ is to guarantee that  $\O_{X^*}$ is a sheaf.)  We say that {\it $X^*$ is a  scheme in $\X(F)$} if the locally ringed space $(X^*,\O_{X^*})$ is a scheme, and that {\it $X^*$ is an affine scheme in $\X(F)$} if  $(X^*,\O_{X^*})$ is an affine scheme. 
 Thus $X^*$ is an affine scheme in $\X(F)$ if and only if  the set of all localizations $A(X)_P$, $P$ a nonzero prime ideal of $A(X)$, is $X$.  
The question of whether a subset of $\X(F)$ is an affine scheme is closely connected to the question of whether the intersection of valuation rings in the set is a Pr\"ufer domain with quotient field $F$. For more on this, see \cite{OZR}. 


A necessarily condition for $X$ to be an affine scheme in $\X(F)$ is that $X$ is quasicompact; similarly, for $X$ to be a scheme, $X$ must be locally quasicompact (i.e., every point has a quasicompact neighborhood). The corollary shows that these conditions are also sufficient for sets $X$ of  valuation rings of rank $\leq 1$ with $J(X) \ne 0$.   


\begin{corollary}  \label{scheme}
Let $X $ be a nonempty set  of rank one valuation rings in $\Zar(R)$ with $J(X) \ne 0$. 
 Then $X^*$ is an affine scheme in $\X(F)$ if and only if $X$ is quasicompact; $X^*$ is a scheme in $\X(F)$ if and only $X$ is locally quasicompact.  
%
%

\end{corollary}

\begin{proof} An affine scheme in $\X(F)$ is quasicompact since the prime spectrum of a ring is quasicompact. 
Conversely, if $X$ is quasicompact, then, with $A = A(X)$,  Theorem~\ref{correspond 1} implies that
 $X^* = \{A_P:P \in \Spec(A)\}$, so that $X^*$ is an affine scheme in $\X(F)$.  
 
 Now suppose $X$ is locally quasicompact.  Let $V \in X$, and let $Z$ be a quasicompact neighborhood of $V$ in  $X$. Then there is an open subset $Y$ of $X$ of the form $Y = {\cal U}(x_1,\ldots,x_n) \cap X$, where $x_1,\ldots,x_n \in V$ and $Y \subseteq Z$.
Since $Z$ is  quasicompact with $J(Z) \ne 0$, Corollary~\ref{agree} implies that the Zariski and inverse topologies agree on $Z$.  Thus $Y$ is a Zariski closed subset of $Z$. Since $Z$ is quasicompact,  $Y$ is also quasicompact, and hence $Y$ is an affine scheme that is open in $X$.  This shows that  $X$ is a union of affine open schemes, so that $X$ is a scheme in $\X(F)$. 
Conversely, if $X$ is a scheme in $\X(F)$, then $X$ is a union of open sets that are affine schemes in $\X(F)$. Thus $X$ is a union of  quasicompact open subsets, proving that $X$ is locally quasicompact. 
%
%
%
%
\end{proof}

\section{Proof of Main Theorem}

In this section we 
prove the main theorem of the introduction. In light of Theorem~\ref{correspond 1}, what remains to be shown is  that 
 the one-dimensional Pr\"ufer domains with nonzero Jacobson radical are B\'ezout domains. In fact, we prove more generally that any one-dimensional domain with nonzero Jacobson radical has trivial Picard group.

\begin{theorem} \label{Pic}  If  $A$ is a one-dimensional domain with $J(A) \ne 0$, then every invertible ideal of $A$ is a principal ideal. 
\end{theorem}

\begin{proof}  Let $I$ be an invertible ideal of $A$. We show that $I$ is a principal ideal of $A$.  After multiplying $I$ by a nonzero element of $J(A)$, we can assume $I \subseteq J(A)$.   Since $A$ has Krull dimension one and $I$ is invertible, there exist $0 \ne a,b \in I$ such that $I = (a,b)A$; see \cite[Corollary 4.3]{SV}, or \cite[Theorem 3.1]{Hei} for a more  general result. 
Let $J = J(A)$, and let 
\begin{center}$X = \{M \in \Max(A):(aA:_Ab) \not \subseteq M\}$ \: and \:  $Y =  \{M \in \Max(A):(aA:_Ab) \subseteq M\}$. 
\end{center}  

First we show that there is $e \in A$ such that $e^2-e \in J$ and 
\begin{center}     
$X = \{M \in \Max(A):e \in M\}$ \: and \: $Y = \{M \in \Max(A):1-e \in M\}$.   
\end{center}
Since $A/J$ is a reduced ring of Krull dimension $0$, $A/J$ is a von Neumann regular ring, and hence every  finitely generated ideal of $A/J$ is generated by an idempotent. In order to apply this observation to the image of  $(aA:_A b)$ in $A/J$, we claim that $(aA:_Ab)$ is a finitely generated ideal of $A$.     Since $I = (a,b)A$, we have $(aA:_Ab) = (aA:_F I)$. Since 
 $I$ is invertible, it follows that  $$(aA:_A b)I = (aA:_F I)I = aA.$$ With $I^{-1} = (A:_FI)$, we have $(aA:_A b) = aI^{-1}$. Since $I^{-1}$ is invertible, $I^{-1}$ is a finitely generated $A$-submodule of $F$, and it follows that $(aA:_A b)$ is a finitely generated ideal of $A$.      Therefore, since $A/J$ is a von Neumann regular ring,  there is $f \in (aA:_A b)$ such that $fA + J = (aA:_A b)+J$ and $f^2-f  \in J$. Set $e = 1-f$. Then $e^2 - e \in J$, and, since  
$J$ is contained in every maximal ideal of $A$, it follows that \begin{center}  $X = \{M \in \Max(A):e \in M\}$ \: and \:  $Y = \{M \in \Max(A):f \in M\}$.  
\end{center}

Next,  since $ab \in J$ and $A/\sqrt{abA}$ has Krull dimension 0, it follows that $J = \sqrt{abA}$. (Recall we have assumed that $I \subseteq J$.)   Since  $$ef = e-e^2 \in J = \sqrt{abA} = \sqrt{abJ},$$ there is $k>0$ such that $e^kf^k \in abJ$.   
Let $c = (a-e^k)(b-f^k)$. 
We claim that $I = cA$.  It suffices to check that this equality holds locally. 

Let $M$ be a maximal ideal of $A$. 
Suppose first that $M \in X$.  Then $(aA:_Ab) \not \subseteq M$, so  there exists $d \in A \setminus M$ such that $db \in aA$. It follows that $bA_M \subseteq aA_M$, and hence $IA_M = aA_M$. Thus to show that $IA_M = cA_M$, it suffices to show that $aA_M = cA_M$. 
If $b \not \in M$, then since $d \not \in M$ we have $db \not \in M$. However, with $b \not \in M$,  the fact that $ab \in M$ implies $db \in aA \subseteq M$, a contradiction. 
 Thus $b \in M$.  Now, since $e \in M$, we have $f = 1-e \not \in M$ and hence (because $b \in M$)  we conclude that $b-f^k \not \in M$. This implies that $cA_M = (a-e^k)A_M$. Since $$e^kA_M = e^kf^kA_M \subseteq aJA_M,$$ there is $j \in J$ and $h \in A \setminus M$ such that $he^k = aj$.   Thus $$h(a-e^k) = ha - he^k = ha - aj = (h-j)a.$$  Since $j \in M$ and $h \not \in M$, we have $h-j \not \in M$. From the fact that $h(a-e^k) = (h-j)a$ we conclude that $$cA_M = (a-e^k)A_M = aA_M,$$ which proves the claim that for each 
 $M \in X$, $IA_M = cA_M$.

Now suppose that $M \in Y$, so that $f \in M$.  We show that $IA_M = cA_M$ in this case also.  Since $(aA:_Ab) \subseteq M$, we have $bA_M \not \subseteq aA_M$. Since $I$ is invertible, $IA_M$ is a principal ideal of $A_M$.  
The following standard argument shows that this implies that $IA_M = bA_M$.  Let $z \in I$ such that $IA_M = zA_M$.  
Then there exist $x,y, s, t \in A_M$ such that $a = zx$, $b = zy$ and $z=  as + bt$.  If $x$ is a unit in $A_M$, then $bA_M \subseteq zA_M = aA_M$, a contradiction. Thus $x$ is not a unit in $A_M$. Since $a = zx = (as+bt)x$, we have $a(1-sx)= btx$, with $1-sx$ a unit in $A_M$ since $x \in MA_M$. Therefore, $aA_M \subseteq bA_M$, which proves that $IA_M = bA_M$.  

We have shown that for $M \in Y$, we have $IA_M = bA_M$.  
To complete the proof of the lemma, it suffices to show  that $bA_M = cA_M$.  
The proof proceeds as in the case where $M \in X$.  Since $bA_M \not \subseteq aA_M$, it follows that $a \in M$, and hence $a-e^k \not \in M$.  Thus $cA_M = (b-f^k)A_M$. Also, $$f^kA_M = e^kf^kA_M \subseteq bJA_M,$$ so that 
there is $h \in A \setminus M$ such that $hf^k = bj$ for some $j \in J$. Thus $$h(b-f^k) = hb -bj = b(h-j),$$ with $h,h-j$ units in $A_M$. Therefore, $(b-f^k)A_M = bA_M$, which shows that $$cA_M = (b-f^k)A_M = bA_M = IA_M.$$ This proves that $IA_M = cA_M$ for all maximal ideals $M $ of $Y$. Since $\Max(A) = X \cup Y$,  we conclude that $I = cA$.  
\end{proof} 

\begin{corollary} \label{Pic 2} If  $A$ is a one-dimensional Pr\"ufer domain with $J(A) \ne 0$, then $A$ is a B\'ezout domain.
\end{corollary}

\begin{proof} This follows from Theorem~\ref{Pic} and  the fact that every finitely generated ideal of a Pr\"ufer domain is invertible \cite[Theorem 22.1]{G}. 
\end{proof} 

 A special case of the lemma  in which it is assumed in addition that $A$ is an almost Dedekind domain for which every maximal ideal of $A$ has finite sharp degree was proved by Loper and Lucas \cite[Theorem 2.9]{LL} using different methods.
 
  With Corollary~\ref{Pic 2} and the results of the preceding sections, we can now prove the main theorem from the introduction. 
 
\begin{theorem} \label{new main theorem}
If $X$ is a quasicompact set of rank one valuation rings in $\Zar(F)$ such that $J(X) \ne 0$, then $A(X)$ is a B\'ezout domain of Krull dimension one with quotient field $F$.  
\end{theorem}

\begin{proof}
 By Theorem~\ref{correspond 1}, $A(X)$ is a one-dimensional Pr\"ufer domain with quotient field $F$ and nonzero Jacobson radical. By Corollary~\ref{Pic 2}, $A(X)$ is a B\'ezout domain, which proves the theorem. 
\end{proof}



\begin{remark} \label{last remark} In \cite{ODiv}  we apply the results of this article to rank one valuation overrings of a two-dimensional Noetherian local domain $D$ with quotient field $F$.  We focus on the divisorial valuation overrings of $D$, i.e., the  DVRs that birationally dominate $D$ and are residually transcendental over $D$. It is shown, for example, that if $n$ is  a positive integer, then the subspace  $X$ of $\Zar(F/D)$ consisting of all divisorial  valuation rings that can be reached through an iterated  sequence of at most $n$ normalized quadratic transforms of $D$ is quasicompact. By Corollary~\ref{correspond 1 cor}, 
$A(X)$ is 
an almost Dedekind domain with nonzero Jacobson radical.  
\end{remark}

{\it Acknowledgment.} I thank the referee for helpful comments that improved the presentation of the article and for suggesting the   version of Example 5.7(2) that is included here.

\end{document}